\documentclass[11pt,a4paper,openbib]{article}

\usepackage{latexsym}
\usepackage{graphics}
\usepackage{amssymb}
\usepackage{amsmath}
\usepackage{amsxtra}                                          
\usepackage{epsfig}
\usepackage{amsthm}
\usepackage{hyperref}

\usepackage[latin1]{inputenc}
\usepackage[T1]{fontenc}

\usepackage[english]{babel}

\newcommand{\dd}{\mathrm{d}}

\author{Laura Desideri%
\thanks{Universit\"at T\"ubingen, Mathematisches Institut, Auf der Morgenstelle 10, 72 076 T\"ubingen, Germany. Tel.: 0049 7071 29 77 459. E-mail: \texttt{desideri@math.jussieu.fr}} 
 and Ruben Jakob%
\thanks{Universit\"at T\"ubingen, Mathematisches Institut, Auf der Morgenstelle 10, 72 076 T\"ubingen, Germany. E-mail: \texttt{jakob@mail.mathematik.uni-tuebingen.de}}}

\title{Immersed solutions of Plateau's problem for
piecewise smooth boundary curves with small total curvature}  

\begin{document}

\theoremstyle{plain}
\newtheorem{theorem}{Theorem}
\newtheorem{proposition}{Proposition}
\newtheorem{lemma}{Lemma}
\newtheorem{corollary}{Corollary}

\theoremstyle{definition}
\newtheorem{definition}{Definition}
\newtheorem*{remark}{Remark}

\maketitle

\begin{abstract}
We provide a new proof of the classical result that any closed rectifiable Jordan curve $\Gamma \subset \mathbb{R}^3$ being piecewise of class $C^2$ bounds at least one \emph{immersed} minimal surface of disc-type, under the additional assumption that the total curvature of $\Gamma$ is smaller than 
$6 \pi$. In contrast to the methods due to Osserman~\cite{Osserman}, Gulliver~\cite{Gulliver} and Alt~\cite{Alt1}, \cite{Alt2}, our proof relies on a polygonal approximation technique, using the existence of immersed solutions of Plateau's problem for polygonal boundary curves, provided by the first author's accomplishment~\cite{Desideri} of Garnier's ideas in~\cite{Garnier}.  
\end{abstract}

\section{Main result and introduction}

Given a closed rectifiable Jordan curve $\Gamma$, the classical solution to Plateau's problem obtained by Douglas~\cite{Douglas} and Rad\'o~\cite{Rado} in the early 1930's yields the existence of a \emph{generalized} minimal surface of disc-type spanning $\Gamma$: its interior is an immersed surface, except possibly at a finite number of branch points. It has been a famous and long out-standing problem whether in fact such branch points actually occur. Entirely immersed solutions to Plateau's problem were finally achieved in the 1970's by Osserman~\cite{Osserman}, Gulliver~\cite{Gulliver} and Alt~\cite{Alt1}, \cite{Alt2}. In this paper we provide a new proof of this now classical result:

\begin{theorem}  \label{Main theorem}
Any closed rectifiable Jordan curve $\Gamma \subset \mathbb{R}^3$ 
being piecewise of class $C^2$ and with total curvature smaller than 
$6 \pi$ bounds at least one \emph{immersed} minimal surface of disc-type. 
\end{theorem} 

Here a minimal surface of disc-type is a (non-constant) map 
$X:\bar B \rightarrow \mathbb{R}^3$,
defined on the closure of the unit disc $B=\{w = (u,v) \in \mathbb{R}^2 \mid \  |w| <1 \}$, of class $C^0(\bar B,\mathbb{R}^3)\cap C^2(B,\mathbb{R}^3)$ which is harmonic and conformally parametrized on $B$, i.e. which meets
\begin{displaymath}
\triangle X = 0 \quad \textnormal{and} \quad 
|X_u|^2-|X_v|^2 =0= \langle X_u,X_v \rangle \quad \textnormal{on } B.
\end{displaymath}
It is termed \emph{immersed} if it additionally satisfies $|X_u|>0$ on $B$,
i.e. if it does not have any branch point on $B$ (see Definition  
\ref{branch points}).
Moreover for any given closed Jordan curve $\Gamma \subset \mathbb{R}^3$ we shall abbreviate by $\mathcal{C}(\Gamma)$ the set of those surfaces 
$X \in C^0(\bar B,\mathbb{R}^3) \cap H^{1,2}(B, \mathbb{R}^3)$ which are \emph{bounded} by $\Gamma$, thus whose restrictions $X\mid_{\partial B}$ map $\partial B$ continuously and weakly monotonically onto $\Gamma$ with mapping degree $1$. 
Finally a closed Jordan curve $\Gamma \subset \mathbb{R}^3$ is termed ``piecewise of class $C^2$'' if there is a homeomorphic 
parametrization $\gamma :[0,2\pi)  \stackrel{\cong }
\longrightarrow \Gamma $ of $\Gamma $ which is twice continuously
differentiable in every point of $[0,2\pi)$ with the exception 
of at most finitely many points $0\leq t_1< \ldots <t_{m}<2\pi$
in which there still exist the one-sided derivatives 
$\dot \gamma (t_i+) := \lim_{t \searrow t_i} \frac{\gamma(t)-\gamma(t_i)}
{t-t_i}$ and $\dot \gamma (t_i-) := \lim _{t \nearrow t_i} 
\frac{\gamma(t_i)-\gamma(t)}{t_i-t}$.

\vspace{0.3cm}

The central tool of our proof of Theorem~\ref{Main theorem}
is the following existence result for Plateau's problem in the 
case of polygonal boundary curves, which could finally be achieved by the first author in \cite{Desideri} (see the Main Theorem on p. 9), by accomplishing Garnier's examination~\cite{Garnier} of second-order Fuchsian differential equations defined on the Riemann sphere $\mathbb{C} \cup \{\infty\}$: the method relies on the resolution of the Riemann--Hilbert problem and on isomonodromic deformations of first-order Fuchsian systems, which are given by the Schlesinger system.  

\begin{theorem}  \label{Lauras_Theorem}
Let $P \subset \mathbb{R}^3$ be a closed polygon in generic position, 
with vertices $A_1,\ldots,A_{N+3}$ ($N>0$). Then there is some 
\emph{immersed} minimal surface of disc-type which is bounded by 
$P$ and maps the three points $-1,-i,1$ onto the last three vertices 
$A_{N+1}, A_{N+2}, A_{N+3}$ of $P$. 
\end{theorem} 

Here by \emph{generic position}, we mean that the $N+3$ edge directions of the polygon $P$ should satisfy the two following conditions: any two directions  may not be parallel, and any three directions may not be coplanar.

In this paper, we have hence to show that one can carry over the above existence statement from such polygonal boundary curves to piecewise $C^2$--smooth closed Jordan curves $\Gamma \subset \mathbb{R}^3$ with total curvature smaller than $ 6 \pi$, by appropriately approximating such a curve $\Gamma$ by polygons $P^n$ in generic positions. We should note here that in \cite{Garnier} Garnier tried to use such a polygonal approximation process as well in order to obtain for any rectifiable, piecewise $C^2$--smooth, closed Jordan curve $\Gamma \subset \mathbb{R}^3$, which is merely assumed to have \emph{finite} total curvature, the existence of some minimal surface $X^* \in \mathcal{C}(\Gamma)$ as a limit of (indeed immersed) minimal surfaces $X^n \in \mathcal{C}(P^n)$ with uniformly converging Weierstrass data $(G^n,H^n)$. Unfortunately, Garnier could only prove that the Weierstrass data $(G^n,H^n)$ of the minimal surfaces $X^n \in \mathcal{C}(P^n)$ converge in $C^0_{loc}(B)$,
 i.e. uniformly on  every interior subdomain $B'\subset \subset B$. Consequently he failed to explain in \cite{Garnier}, pp. 116--144, why this limit surface $X^*$ should map $\partial B$ continuously and weakly monotonically onto $\Gamma$ with mapping degree $1$, i.e. why $X^*$ should indeed be bounded by the approximated curve $\Gamma$. This could be one of the reasons for the fact that his paper \cite{Garnier} has never been accepted as the first rigorous solution of Plateau's problem for arbitrary piecewise smooth boundary curves. Moreover the limit surface $X^*$ of some sequence of minimal immersions $X^n \in \mathcal{C}(P^n)$ whose Weierstrass data converge merely in $C^0_{loc}(B)$ might in general have branch points in $B$. In fact, $X^*$ could even be a constant map, thus a completely degenerated surface, if there is no further knowledge available neither about the involved boundary curves $\Gamma$ and $P^n$, nor about the converging minimal immersions $X^n$.

In this paper, we at first reparametrize the minimal immersions $X^n \in \mathcal{C}(P^n)$ appropriately in order to obtain the existence of some subsequence $\{\tilde X^{n_l}\}$ whose boundary values can be shown to converge uniformly to a continuous and weakly monotonic parametrization $\beta$ of the curve $\Gamma$. We then conclude in a second step that the harmonic extension $X^*$ of $\beta$ onto $\bar B$ does not only inherit the conformality but also the absence of branch points from the uniformly converging minimal immersions $\tilde X^{n_l}$, by means of two Theorems due to Sauvigny \cite{Sauv5},  \cite{Sauv1}. 

In contrast to this approximative strategy Osserman, Gulliver and Alt considered some global minimizer $X^*$ of the Dirichlet energy $\mathcal{D}(X)= \frac{1}{2} \int_B \mid DX \mid^2 \dd u\dd v$ in the class $\mathcal{C}(\Gamma)$ and classified the branch points of $X^*$ into two different types, into ``true'' and ``false'' branch points. Alt showed firstly in \cite{Alt1} by a certain surgery technique that any ``true'' branch point $w_0 \in B$ of $X^*$ would lead to an impossible behaviour of the unit normal of $X^*$ in a neighbourhood of $w_0$. After that he used in \cite{Alt2} another surgery technique in order to even rule out the existence of eventually remaining ``false'' branch points of any such global minimizer $X^*$ of $\mathcal{D}$ (in fact of any functional of some class which contains $\mathcal{D}$ in particular). Since both of his exclusion arguments work for global minimizers in $\mathcal{C}(\Gamma)$ for any closed rectifiable Jordan curve in $\mathbb{R}^3$, 
 his result is obviously stronger than the above Theorem \ref{Main theorem}, to be proved in Section \ref{Proof_of_Main_Th}.

\section{Preparations for the proof of Theorem \ref{Main theorem}}
\label{Preparations} 

Firstly we need the following definitions.

\begin{definition}  \label{poly approx}
Let $\Gamma$ be an arbitrary closed rectifiable Jordan curve in $\mathbb{R}^3$ being piecewise of class $C^2$ and $\gamma:\mathbb{S}^1 \stackrel{\cong} \longrightarrow \Gamma $ a fixed parametrization of $\Gamma$.
\begin{itemize}
\item[(i)] We term the elements of a sequence $\{P^n\}$ of simple closed polygons $P^n \subset \mathbb{R}^3$ in generic positions  
\emph{polygonal approximations} of $\Gamma$ if there
exist homeomorphisms $\varphi^n:\Gamma \stackrel{\cong }\longrightarrow P^n$
such that 
\begin{equation}  \label{varphin conv}
\max_{x\in \Gamma } \mid x - \varphi^n(x) \mid \longrightarrow 0
\qquad  \text{as } n\rightarrow \infty.
\end{equation}

\item[(ii)] We define the total curvature $\textnormal{TC}(P)$ of some polygonal approximation $P$ as the sum of the exterior angles $\eta_1, \ldots, \eta_M$ of $P$ at its consecutive vertices $A_1, \ldots, A_M$.
\item[(iii)] Finally, let $L:=\mathcal{L}(\Gamma)$ denote the length of $\Gamma$ and consider the piecewise smooth parametrization $\gamma:[0,L)  \stackrel{\cong }\longrightarrow \Gamma$
of $\Gamma $ with $\gamma\in C^2([0,L) \setminus \{t_1,\ldots, t_m\})$ and $\mid \dot \gamma \mid \equiv 1$ on $[0,L) \setminus \{t_1,\ldots, t_m\}$, for some subdivision $0 \leq t_1< t_2<\ldots <t_{m} < L$ of $[0,L)$. We term  $\vartheta_i \in (0,\pi)$ the smaller angle between the two tangent vectors 
$\dot \gamma (t_i+)$ and $\dot \gamma (t_i-)$ and define the total curvature of $\Gamma$ as
\begin{displaymath}
\textnormal{TC}(\Gamma) := \sum_{i=1}^m \vartheta_i + 
\sum_{i=1}^m \int_{t_i}^{t_{i+1}} \mid \ddot{\gamma} \mid \,\dd s,  
\end{displaymath}
where we set $t_{m+1}:=t_1$.
\end{itemize}  
\end{definition}

Now we can state the following technical approximation tool: 

\begin{proposition}  \label{Approx of Gamma}
Let $\Gamma$ be an arbitrary closed rectifiable Jordan curve in 
$\mathbb{R}^3$ being piecewise of class $C^2$ with three fixed different  points $x_0,x_1,x_2$, and let $\varepsilon>0$ be arbitrarily fixed. Then there exist some sequence $\{P^n\}$ of polygonal approximations of $\Gamma$ with homeomorphisms $\varphi^n:\Gamma \stackrel{\cong }\longrightarrow P^n$ 
and some integer $N(\varepsilon) \in \mathbb{N}$ such that there hold
\begin{align}
\mathcal{L}(P^n) &< \mathcal{L}(\Gamma) + \varepsilon, \qquad \textnormal{and} \label{lenght conv}\\
\textnormal{TC}(P^n) &< \textnormal{TC}(\Gamma) + \varepsilon,  \label{curvature conv}
\end{align}
for every $n>N(\varepsilon)$. In addition, the sequence $\{P^n\}$ can be constructed in such a way that three vertices $a^n_0, a^n_1, a^n_2$ of each polygonal approximation $P^n$ coincide with the fixed points $x_0,x_1,x_2$ of $\Gamma$, i.e. such that
$a^n_0 = x_0, a^n_1 = x_1$ and $a^n_2 = x_2$ for each $n\in \mathbb{N}$,
without violating the property \emph{to be in generic position} of each 
$P^n$. Moreover, the homeomorphisms $\varphi^n$ can then be chosen such that  $\varphi^n(x_k)=x_k$ $(k=0,1,2)$ for every $n\in \mathbb{N}$.
\end{proposition}

Moreover, in the next section we will use the following two notions of \emph{branch points} of some minimal surface with arbitrary continuous boundary values:
 
\begin{definition}  \label{branch points}
Let $X \in C^0(\bar B,\mathbb{R}^3)\cap C^2(B,\mathbb{R}^3)$ be 
an arbitrary minimal surface.
\begin{itemize}
\item[(i)] We call a point $w_0 \in B$ an \emph{interior branch point} 
of $X$ if $|X_u(w_0)|=0$.
\item[(ii)] We call a point $w_0 \in \partial B$ a \emph{boundary branch point} 
of $X$ if there holds
\begin{equation}  \label{boundary branch point}
|X_u(w_k)| \longrightarrow 0  \quad \text{as } k \rightarrow \infty
\end{equation}
for any sequence $\{w_k\} \subset B$ converging to $w_0$.
\end{itemize}
\end{definition} 

For any minimal surface $X \in \mathcal{C}(P)$ bounded by some closed polygon $P$, Heinz proved in \cite{Heinz} in particular the following asymptotic expansion for the complex derivative $X_w := \frac{1}{2} \big{(} \frac{\partial}{\partial u} - i \frac{\partial}{\partial v} \big{)} X$ about the boundary points $e^{i t_1}, \ldots, e^{i t_{M}}$ that are mapped by $X$ onto the vertices $A_1, \ldots, A_M$ of $P$: 
\begin{equation}  \label{expan of Xw}
X_w(w) = c_k  \,(w - e^{i t_k})^{\rho_k + m_k} 
+ O(\mid w - e^{i t_k} \mid ^{\rho_k + m_k + \epsilon_k }) ,
\end{equation}
for each $k \in \{1,\ldots, M\}$. Here, the $c_k$ are fixed vectors in $\mathbb{C}^3\setminus \{0\}$, the $m_k$ are non-negative integers -- the so-called \emph{branch point orders} of $X$ with respect to the points $e^{i t_k}$ (see Definition \ref{Branchorder} (i) below) --, $\rho_k \in (-1,0)$ and 
$\epsilon_k \in (0,1)$. The angle of the minimal surface at the vertex $A_k$ is $(\rho_k + m_k+1)\pi$. If the integer $m_k$ is even (resp. odd), then $|\rho_k|\pi$ is the exterior angle $\eta_k$ 
(resp. the interior angle) of the polygon $P$ at the vertex $A_k$.
Now one can easily see that such a point $e^{i t_k}$ is a 
\emph{boundary branch point} of $X$ in the sense of Definition \ref{branch points} if and only if its order $m_k$ is a \emph{positive} integer. \\\\
Moreover, about any point $w_0\in \bar B\setminus \{e^{it_k}\}_{k=1}^{M}$ there is a Taylor expansion of $X_w$, i.e. 
\begin{equation}  \label{Taylor}
X_w(w)= a_{m(w_0)}\, (w-w_0)^{m(w_0)} + a_{m(w_0)+1} \, (w-w_0)^{m(w_0)+1}+ \ldots,
\end{equation}
where $a_{m(w_0)}\in \mathbb{C}^3 \setminus \{0\}$ and $m(w_0) \in \mathbb{N}_0$. 
Thus again, any point $w_0 \in\bar B\setminus \{e^{it_k}\}_{k=1}^{M}$ is a branch point of the  
minimal surface $X$ -- in the sense of Definition \ref{branch points} --
if and only if the integer $m(w_0)$ in (\ref{Taylor}) is positive. This motivates the following  

\begin{definition}  \label{Branchorder}
(i) We term the exponent $m_k$ in (\ref{expan of Xw}) resp. $m(w_0)$ in
(\ref{Taylor}) the branch point order of the surface
$X$ at the point $e^{i t_k}$, $k=1,\ldots, M$, resp. at the point
$w_0\in \bar B\setminus \{e^{i t_k}\}_{k=1}^{M}$.\\
(ii) We define the total branch point order of $X$ by 
\begin{equation} \label{total_branch_point_order}
\kappa (X):=\sum_{w\in B} m(w) + 
\frac{1}{2} \sum_{w\in \partial B \setminus \{e^{i t_k}\}_{k=1}^{M}}  m(w)
+\frac{1}{2} \sum_{k=1}^M m_k, 
\end{equation}  
which is a finite sum since $X$ can only have isolated and thus 
finitely many branch points on $\bar B$ on account of the expansions 
(\ref{expan of Xw}) and (\ref{Taylor}). 
\end{definition} 
The geometric meaning of the exponents $\rho_k$, k=$1, \ldots, M$, 
appearing in (\ref{expan of Xw}), becomes even clearer by means of the 
following Gauss-Bonnet formula for minimal surfaces with polygonal 
boundaries, the main result of Sauvigny's article \cite{Sauv5}: 

\begin{proposition} \label{Sauvignys_Gauss_Bonnet}
Let $X \in \mathcal{C}(P)$ be a minimal surface bounded by some closed 
polygon $P$, with vertices $A_1, \ldots, A_M$, having Gauss-curvature $K_{X}$. 
Then there holds the formula 
\begin{equation} \label{Sauvignys_Gauss_Bonnet_formula}
\int_B \mid K_{X} \mid E \, \dd u\dd v + 2 \pi (1+\kappa(X)) 
= \pi \, \sum_{k=1}^M \mid \rho_k \mid,
\end{equation}
where $E$ denotes $|X_u|^2 \equiv |X_v|^2$.
\end{proposition}

For later purpose we should state the following 
\begin{corollary} \label{Special_Gauss_Bonnet}
Let $X \in \mathcal{C}(P)$ be a minimal surface bounded by some closed 
polygon $P$ which has only branch points in the points $e^{i t_k}$, 
$k=1, \ldots, M$. Then there holds:
\begin{equation} \label{Special_Gauss_Bonnet_formula}
 \int_B \mid K_{X} \mid E \, \dd u\dd v + 2\pi 
= \pi \sum_{k=1}^M  (\mid \rho_k \mid - m_k).
\end{equation}
Moreover, letting $\eta_k$ denote the $M$ exterior angles of $P$, 
$s$ the number of branch points of $X$ and $k_1, k_2, \ldots, k_{M-s}$ 
those indices in $\{1, \ldots, M\}$ for which $e^{i t_{k_j}}$ is not 
a branch point of $X$, then we obtain the estimate
\begin{equation} \label{Special_Gauss_Bonnet_estimate}
\int_B \mid K_{X} \mid E \, \dd u\dd v
\leq \sum_{j=1}^{M-s} \eta_{k_j}  - 2\pi 
\end{equation}
and in particular 
\begin{equation} \label{Special_Gauss_Bonnet_estimate2}
 \int_B \mid K_{X} \mid E \, \dd u\dd v  
\leq \textnormal{TC}(P) - 2\pi,
\end{equation}
where $\textnormal{TC}(P):= \sum_{k=1}^M \eta_k$ is
the total curvature of $P$.
\end{corollary}
\begin{proof}
Formula (\ref{Special_Gauss_Bonnet_formula}) follows immediately from 
formula (\ref{Sauvignys_Gauss_Bonnet_formula}) and the definition 
of the total branch point order $\kappa(X)$ in (\ref{total_branch_point_order}). 
Now, as explained above, $\pi |\rho_k|$ measures exactly the exterior angle 
$\eta_k$ of the polygon $P$ at its vertex $A_k$ if $m_k$ is even, 
in particular if $m_k=0$, i.e. if the point $e^{it_k}$ is not a
branch point of $X$. This yields 
$\pi (|\rho_{k_j}| - m_{k_j}) = \eta_{k_j}$ for $j=1, \ldots, M-s$ 
on the one hand. On the other hand, 
if some point $e^{it_k}$ is a branch point of $X$, then we have $m_k \geq 1$
and thus $\mid \rho_k \mid - m_k \leq \mid \rho_k \mid - 1 < 0$
on account of $\rho_k \in (-1,0)$. In combination with formula 
(\ref{Special_Gauss_Bonnet_formula}) this yields the claimed 
estimate (\ref{Special_Gauss_Bonnet_estimate}) which instantly
implies estimate (\ref{Special_Gauss_Bonnet_estimate2}) as well.
\end{proof}

Finally we need the following compactness result for boundary values, which 
we shall prove for the sake of completeness (see also \cite{Nit}, 
Paragraphs 21, 234 and 235):
 
\begin{proposition}   \label{Compactness}
Let $\Gamma $ and $\{P^n\}$ be as in Proposition \ref{Approx of Gamma} and 
$X^n \in \mathcal{C}(P^n)$ some sequence of surfaces with uniformly bounded Dirichlet energies, i.e. $\mathcal{D}(X^n)\leq M$ for every $n\in \mathbb{N}$, and satisfying a \emph{uniform three-point-condition} 
$X^n(e^{i\frac{\pi }{2}(2+k)}) = x_k$, for $k=0,1,2$, where $x_0,x_1,x_2$ are three fixed consecutive points on $\Gamma$. Then there exists some subsequence $\{X^{n_l}\}$ whose boundary values satisfy
\begin{displaymath}
X^{n_l}\mid _{\partial B}\longrightarrow \beta \qquad \text{in }
C^0(\partial B, \mathbb{R}^3),
\end{displaymath}
where $\beta:\mathbb{S}^1\longrightarrow \Gamma$ is a continuous, weakly monotonic map onto $\Gamma$ with mapping degree $1$ and 
$\beta (e^{i\frac{\pi }{2}(2+k)})=x_k$, for $k=0,1,2$.    
\end{proposition} 

\begin{proof}
We consider a fixed homeomorphic parametrization 
$\gamma :\mathbb{S}^1 \stackrel{\cong }
\longrightarrow \Gamma $ of $\Gamma $ and the weakly monotonic maps\,
$(\varphi^n)^{-1}\circ X^n\mid _{\partial B}:\partial B\longrightarrow \Gamma $
onto $\Gamma $. For each $n\in \mathbb{N}$ there exist non-decreasing maps 
$\sigma^n :[0,2\pi ] \longrightarrow [0,4\pi )$, with 
$\sigma^n(2\pi )=\sigma^n(0)+2\pi $, such that 
$(\varphi^n)^{-1}\circ X^n(e^{it})=\gamma (e^{i\sigma^n(t)})$ for all $t\in [0,2\pi ]$. By (\ref{varphin conv}) we conclude that
\begin{align}  \label{concluedeo}
\max_{t\in [0,2\pi ]} \left| 
\gamma \left(e^{i\sigma^n(t)}\right)-X^n\left(e^{it}\right) \right| 
&=\max_{t\in [0,2\pi ]} \left| \gamma \left(e^{i\sigma^n(t)}\right)- 
\varphi^n\left(\gamma \left(e^{i\sigma^n(t)}\right)\right) \right| \nonumber\\  
&=\max_{x\in \Gamma } \left| x - \varphi^n(x) \right| \longrightarrow 0 \qquad 
\text{as }n\rightarrow \infty.
\end{align}
Furthermore, Helly's selection principle (see \cite{Natan}, p. 248) yields 
some subsequence $\{\sigma^{n_l}\}$ and a non-decreasing function $\sigma$ on 
$[0,2\pi ]$ such that 
\begin{equation}  \label{conv of sigma n}
\sigma^{n_l}(t)\longrightarrow  \sigma (t) \qquad \forall \,t \in [0,2\pi ],
\qquad \text{as } l\rightarrow \infty,   
\end{equation}
thus also\, $\gamma (e^{i\sigma^{n_l}(t)}) \longrightarrow 
\gamma (e^{i\sigma (t)})$ for all $t\in [0,2\pi ]$. Hence, together 
with (\ref{concluedeo}) we arrive at
\begin{equation}  \label{pointw conv of Xnl}
X^{n_l}\left(e^{it}\right) \longrightarrow  \gamma \left(e^{i\sigma (t)}\right) \qquad \forall \,t\in [0,2\pi ], \qquad \text{as } l\rightarrow \infty,
\end{equation}
which especially implies 
$\gamma (e^{i \sigma( \frac{\pi }{2}(2+k))}) = x_k$, for $k=0,1,2$, due to the required uniform three-point-condition imposed on  
$X^n\mid_{\partial B}$, for each $n$.  Hence, since $x_i\not = x_j$\, for $i\not=j$ we see that 
\begin{equation}  \label{three diff}
\sigma \left(\frac{\pi }{2}(2+i)\right) \not 
= \sigma \left( \frac{\pi }{2}(2+j) \right)  \quad \mod \,2\pi, \qquad \text{for }  i\not=j.
\end{equation}
Now an extension of Helly's selection principle (see \cite{Polya}, 
p. 63 and p. 226) provides the uniform convergence of the $\sigma^{n_l}$
if $\sigma$ is known to be continuous, what we are going to prove now.
To this end we shall assume that $\sigma $ was not continuous.   
Since $\sigma $ is weakly monotonic, there exist the one-sided limits\,
$\sigma (t+0)$ and $\sigma (t-0)$, for all $t\in [0,2\pi ]$, where
we mean\, $\sigma (0-0):=\sigma (2\pi-0)-2\pi$ and\, 
$\sigma (2\pi +0):=\sigma (0+0)+2\pi$. The points of discontinuity of 
$\sigma $ coincide with those points $t^*$ in which we have\,     
$0<\sigma (t^*+0)-\sigma (t^*-0)$. Moreover there holds\, 
$\sigma (t^*+0)-\sigma (t^*-0)<2\pi$, otherwise on account of the 
monotonicity of $\sigma$ and\, $\sigma (2\pi )=\sigma (0)+2\pi$ we would 
have \,$\sigma (t)\equiv \sigma (t^*-0)$ \, on $[0,t^*)$ and\, 
$\sigma (t)\equiv \sigma (t^*+0)$ \, on $(t^*,2\pi ]$, which contradicts       
(\ref{three diff}). Hence, we conclude that\, 
$\sigma (t^*+0)\not=\sigma (t^*-0)$ \, mod $2\pi$ and therefore
by the injectivity of $\gamma$ 
\begin{equation}  \label{jump}
\gamma \left(e^{i\sigma(t^*+0)}\right) \not = \gamma \left(e^{i\sigma(t^*-0)}\right)
\end{equation}   
in every discontinuity point $t^*$ of $\sigma $. Now we fix such a point
$t^*$ which we may suppose to be contained in $(0,2\pi )$ without loss of generality. By (\ref{jump}) we have\, 
$\mid \gamma (e^{i\sigma(t^*+0)}) - \gamma (e^{i\sigma(t^*-0)}) \mid =\epsilon>0$ for some $\epsilon>0$. Moreover by the existence of the one-sided limits $\sigma (t^*+0)$,\, $\sigma (t^*-0)$ and by the continuity of $\gamma $
there is some sufficiently small $\alpha>0$ such that 
$[t^*-\alpha, t^*+\alpha] \subset (0,2\pi )$ and 
\begin{align*}
&\left| \gamma \left(e^{i\sigma(t)}\right) - \gamma \left(e^{i\sigma(t^*-0)}\right) \right| < 
\frac{\epsilon }{3} \qquad \forall \,t\in (t^*-\alpha, t^*) \\ 
\textnormal{and} \qquad 
&\left| \gamma \left(e^{i\sigma(t)}\right) - \gamma \left(e^{i\sigma(t^*+0)}\right) \right|
<\frac{\epsilon }{3} \qquad \forall \,t\in (t^*, t^*+\alpha),  
\end{align*}          
which implies together with (\ref{pointw conv of Xnl}):
\begin{equation}  \label{large jump}
\lim_{l\rightarrow \infty } \left| X^{n_l}\left(e^{it'}\right) - X^{n_l}\left(e^{it''}\right) \right|   
= \left| \gamma \left(e^{i\sigma(t')}\right) - \gamma \left(e^{i\sigma(t'')}\right) \right|
>\frac{\epsilon }{3}
\end{equation}
for all $t'\in (t^*-\alpha , t^*)$ and all $t''\in (t^*, t^*+\alpha )$.
Now we only consider pairs $t'$, $t''$ such that $0<t''-t^*=t^*-t'<\alpha $. 
For $r:=2 \sin\big{(} \frac{t^*-t'}{2} \big{)}$ we have\, 
$\partial B_{r}(e^{it^*}) \cap \partial B = \{e^{it'}, e^{it''}\}$. 
We introduce the notation 
$\{w_1(\rho ), w_2(\rho )\} := \partial B_{\rho }(e^{it^*}) \cap \partial B$, for $0<\rho < 2 \sin \big{(}\frac{\alpha }{2} \big{)}$.  
Now making use of the requirement $\mathcal{D}(X^n)\leq M$ for all
$n\in \mathbb{N}$, and of H\"older's inequality one easily infers 
from Fatou's lemma that \,$\liminf_{l\rightarrow \infty } 
\mid X^{n_l}(w_1(\rho )) - X^{n_l}(w_2(\rho )) \mid^2 \frac{1}{\rho }
\in L^1([\delta ,\sqrt{\delta}])$ for any 
$\delta < 4 \sin^2 \big{(} \frac{\alpha }{2} \big{)}$ and that there holds
(see \cite{Nit}, pp. 207--209):
\[
\frac{1}{2\pi } \int_{\delta}^{\sqrt{\delta }} 
\liminf_{l\rightarrow \infty } 
\mid X^{n_l}(w_1(\rho )) - X^{n_l}(w_2(\rho )) \mid^2 \frac{ \dd\rho}{\rho } 
\leq M.        
\]
Combining this with (\ref{large jump}) we achieve:
\[
M> \frac{\epsilon ^2}{18 \pi } \int_{\delta}^{\sqrt{\delta }} 
\frac{ \dd\rho}{\rho } 
= \frac{\epsilon ^2}{36 \pi } \log \Big{(} \frac{1 }{\delta } \Big{)} \qquad 
\forall \,\delta < 4 \sin^2 \Big{(} \frac{\alpha }{2} \Big{)},
\]
which yields a contradiction for a sufficiently small choice of $\delta$.   
Hence, $\sigma$ must be continuous on $[0,2\pi ]$ and therefore 
the convergence in (\ref{conv of sigma n}) even uniform:
\[
\sigma^{n_l}\longrightarrow  \sigma \qquad \text{in } C^0([0,2\pi ]). 
\]   
As $\gamma$ is uniformly continuous on $\mathbb{S}^1$ this yields
\[
\gamma \left(e^{i\sigma^{n_l}(\,\cdot \,)}\right) \longrightarrow  
\gamma \left(e^{i\sigma(\,\cdot \,)}\right)  \qquad \text{in }
C^0([0,2\pi ],\mathbb{R}^3), 
\]
and together with (\ref{concluedeo}) we finally arrive at
\begin{equation}  \label{unif conve}
X^{n_l}\left(e^{i\,(\,\cdot \,)}\right) \longrightarrow  
\gamma \left(e^{i\sigma (\,\cdot \,)}\right)
\qquad \text{in } C^0([0,2\pi ],\mathbb{R}^3).
\end{equation}
Hence, defining $\beta:\mathbb{S}^1\longrightarrow \Gamma $ via \,
$\beta(e^{i\,(\,\cdot \,)}):=\gamma (e^{i\sigma (\,\cdot \,)})$ we see that 
$\beta $ has in fact the asserted properties on account of the continuity 
and weak monotonicity of $\sigma$ and of its property 
$\sigma (2\pi) =\sigma (0) +2\pi$, in combination with the 
homeomorphy of $\gamma$.
Finally $\beta (e^{i\frac{\pi }{2}(2+k)})=x_k$, for $k=0,1,2$, follows immediately from (\ref{unif conve}).
\end{proof}

\section{Proof of Theorem \ref{Main theorem}}  
\label{Proof_of_Main_Th}

Now we fix some closed rectifiable, piecewise $C^2$--Jordan curve $\Gamma$  
and choose three different consecutive points $x_0,x_1,x_2$ on $\Gamma$.  
By Proposition \ref{Approx of Gamma} we obtain some sequence $\{P^n\}$ of polygonal approximations of $\Gamma$ with $N_n+3$ vertices, which we shall
enumerate in the following manner:  
\begin{equation}  \label{vertices} 
(a^n_0,A^n_1,\ldots,A^n_{l_n};a^n_1;A^n_{l_n+1} ,\ldots, A^n_{m_n};
a^n_2;A^n_{m_n+1},\ldots, A^n_{N_n}),
\end{equation} 
such that $a^n_0 = x_0, a^n_1 = x_1$ and $a^n_2 = x_2$ 
for each $n\in \mathbb{N}$. Now, since each polygonal approximation 
$P^n$ is in generic position, Theorem \ref{Lauras_Theorem}
guarantees the existence for each $n$ of some immersed minimal surface $X^n$ spanning $P^n$, and mapping the three points $-1,-i,1$ onto the last three vertices in~\eqref{vertices}. To establish Theorem~\ref{Main theorem}, we are now going to successively apply Proposition~\ref{Compactness} and a theorem due to Sauvigny (Theorem~1, (ii) in \cite{Sauv1}) to a sequence $\{\tilde X^n\}$ of immersions obtained from $\{X^n\}$ by an appropriate reparametrization.

\bigskip

Since the boundary values of the surfaces $X^n$ are known to map $\partial B$ (weakly) monotonically onto 
$P^n$ with mapping degree $1$ we can estimate the Dirichlet energies, resp. the areas, of the $X^n$ by means of the isoperimetric inequality and (\ref{lenght conv}):
\begin{equation}  \label{Isoperimetric} 
\mathcal{D}(X^n) = \mathcal{A}(X^n) \leq 
\frac{1}{4 \pi} \textnormal{Tot.Var.}(X^n \mid_{\partial B})^2
= \frac{1}{4 \pi} \mathcal{L}(P^n)^2
\leq c\, \mathcal{L}(\Gamma)^2, 
\end{equation} 
for each $n \in \mathbb{N}$, where $c$ is a positive constant. Moreover there is a unique biholomorphic automorphism $\Phi^n: B \stackrel{\cong} \longrightarrow B$, with a unique homeomorphic extension onto $\bar B$, such that the reparametrization $\tilde X^n:= X^n \circ \Phi^n$ maps the three points $-1,-i,1$ onto the three specified vertices $a^n_0 = x_0, a^n_1 = x_1$ and $a^n_2 = x_2$ of $P^n$, which are fixed on $\Gamma$ as $n \rightarrow \infty$. As the automorphism $\Phi^n$ is biholomorphic on $B$ and as its extension $\Phi^n\mid_{\partial B}: \partial B \stackrel{\cong} \longrightarrow \partial B$ performs an orientation preserving homeomorphism, the reparametrized surface $\tilde X^n$ is again an immersed minimal surface spanning $P^n$ which in addition meets the \emph{uniform three-point-condition} of Proposition \ref{Compactness}. And due to $\mathcal{D}(\tilde X^n) = \mathcal{D}(X^n)$ the surface $\tilde X^n$ clearly also
satisfies the estimate (\ref{Isoperimetric}). Thus we can apply Proposition~\ref{Compactness} in order to obtain the existence of some subsequence $\tilde X^{n_l}$ whose boundary values  
$\tilde X^{n_l}\mid_{\partial B}$ satisfy
\begin{displaymath}
\tilde X^{n_l}\mid _{\partial B}\longrightarrow \beta \qquad \textnormal{in } C^0(\partial B, \mathbb{R}^3),
\end{displaymath}
where $\beta:\mathbb{S}^1\longrightarrow \Gamma$ is a continuous, weakly monotonic map onto $\Gamma$ with mapping degree 1. Now, by the maximum principle and Cauchy's estimates we can immediately conclude that the subsequence $\tilde X^{n_l}$ converges in $C^0(\bar B) \cap C^2_{loc}(B)$ to the unique harmonic extension $X^*$ of $\beta$ onto $\bar B$ which is thus again conformally parametrized and bounded by $\Gamma$. 

\bigskip

Now thanks to Theorem~1, (ii) in \cite{Sauv1}, to prove that the harmonic extension $X^*$ is free of interior branch points, it is sufficient to show that there is a constant $e_0\in(0,4\pi)$ such that there holds for every sufficiently large $l$:
\[
\int_B \mid K_{\tilde X^{n_l}} \mid \tilde E^{n_l} \, \dd u\dd v \leq e_0
\]
with $\tilde E^{n_l} := |\tilde X^{n_l}_u|^2 \equiv |\tilde X^{n_l}_v|^2$. We fix some $n$ arbitrarily. We know by Theorem~\ref{Lauras_Theorem} that $X^n$ maps some $N_n-$tuple of points $e^{it^n_1}, e^{it^n_2}, \ldots, e^{it^n_{N_n}} \in \mathbb{S}^1 \cap \{ \Im(w)> 0 \}$, precisely with $0<t^n_1<t^n_2< \ldots < t^n_{N_n}< \pi$, onto the first $N_n$ vertices of $P^n$ in (\ref{vertices}) and the three points $-1,-i,1$ onto the last three vertices of $P^n$. The set of branch points $\mathcal{B}^n$ of $X^n$, which is a subset of the set $\mathcal{S}^n:= \{e^{it^n_1}, e^{it^n_2}, \ldots, e^{it^n_{N_n}}, -1,-i,1 \}$, consists of $s^n \in \{0, 1, \ldots, N_n+3 \}$ points. For simplification of notation and without loss of generality, we shall suppose that these $s^n$ branch points of $X^n$ are adjacent, e.g. $\mathcal{B}^n = \{e^{it^n_1}, e^{it^n_2}, \ldots, e^{it^n_{s^n}} \}$. We shall term the $N_n+3$ exterior angles in the consecutive vertices (\ref{vertices}) of $P^n$: $\eta^n_1, \eta^n_2, \ldots, \eta^n_{N_n+3}$. Now, applying estimate (\ref{Special_Gauss_Bonnet_estimate}) to each surface $X^n$ with 
branch point set $\mathcal{B}^n = \{e^{it^n_1}, e^{it^n_2}, \ldots, e^{it^n_{s^n}} \}$ we obtain:
\begin{equation}  \label{total_Gauss_curvature1}
\int_B \mid K_{X^n} \mid E^n \, \dd u\dd v 
\leq \sum_{j=s^n+1}^{N_n+3} \eta^n_j - 2\pi,
\end{equation}
$\forall \, n \in \mathbb{N}$. Thus, applying the weaker estimate 
(\ref{Special_Gauss_Bonnet_estimate2}) to each surface $X^n$ 
in combination with the requirement that 
$\textnormal{TC}(\Gamma) = 6 \pi - 2\varepsilon$, for some $\varepsilon>0$, we achieve 
by~\eqref{curvature conv} in Proposition \ref{Approx of Gamma} the existence of some large integer $N(\varepsilon)$ such that there holds 
\begin{equation}  \label{total_Gauss_curvature2}
\int_B \mid K_{X^n} \mid E^n \, \dd u\dd v \leq \textnormal{TC}(P^n) - 2\pi
<\textnormal{TC}(\Gamma) + \varepsilon - 2\pi = 4 \pi - \varepsilon,
\end{equation}
for the integral over the (negative) Gaussian curvature $K_{X^n}$ of $X^n$, whenever $n>N(\varepsilon)$. Finally, since there holds $K_{\tilde X^n} = K_{X^n} \circ \Phi^n$ on $B$ for the Gaussian curvature of $\tilde X^n$ on account of the Theorema Egregium and the biholomorphy of $\Phi^n$ we deduce from (\ref{total_Gauss_curvature2}) that the reparametrized minimal surface $\tilde X^{n}$ again satisfies the estimate 
\begin{align*}  
\int_B \mid K_{\tilde X^n} \mid \tilde E^n \, \dd u\dd v 
& =\int_B \mid K_{X^n} \mid \circ \Phi^n \, E^n\circ \Phi^n \, 
\mid (\Phi^n)' \mid^2 \dd u\dd v\\
& \equiv \int_B \mid K_{X^n}\mid \circ \Phi^n \, E^n\circ \Phi^n \, 
\det (D_{(u,v)} \Phi^n) \, \dd u\dd v \\
& = \int_B \mid K_{X^n} \mid \, E^n \, \dd u\dd v < 4 \pi - \varepsilon,
\end{align*}
whenever $n>N(\varepsilon)$. Thus by Theorem~1, (ii) in \cite{Sauv1} due to Sauvigny we may infer that the limit surface $X^*$ in fact inherits the absence of interior branch points of the converging minimal surfaces $\tilde X^{n_l}$, i.e. that $X^*$ is free of branch points on the open unit disc $B$, just as asserted in Theorem~\ref{Main theorem}.\\\\\\
{\bf Acknowledgements}\,\,
The authors would like to thank Prof. Dr. Ulrich Dierkes for having motivated this research, and Beno\^it Kloeckner for fruitful comments about the applied approximation technique.

\end{document}